\documentclass{amsart} 

\usepackage{lmodern}
\usepackage{microtype}
\usepackage[english]{babel} 
\usepackage{hyperref} 
\usepackage{mathrsfs}
\usepackage{mathtools}
\usepackage{xfrac}
\usepackage{xcolor}
\usepackage[msc-links]{amsrefs} 
\usepackage{tikz}
\usepackage{pgfplots} 
\pgfplotsset{compat=1.18}
\usepackage{tikz-cd} 
\tikzcdset{arrow style=tikz, arrows={thick}, diagrams={>=stealth}}

\theoremstyle{plain} 
\newtheorem{theorem}{Theorem}
\newtheorem{lemma}[theorem]{Lemma} 
\newtheorem{corollary}[theorem]{Corollary}

\theoremstyle{definition}
\newtheorem*{cec}{Critical exponent conjecture}
\newtheorem*{weakcec}{Weak critical exponent conjecture}
\newtheorem*{conjecture}{Conjecture}

\begin{document} 
\title{Critical exponents of the Riesz projection}
\date{\today} 

\author{Ole Fredrik Brevig} 
\address{Department of Mathematics, University of Oslo, 0851 Oslo, Norway} 
\email{obrevig@math.uio.no}

\author{Adri\'{a}n Llinares}
\address{Department of Mathematics and Mathematical Statistics, Ume{\aa} University, SE-90187 Ume{\aa}, Sweden}
\curraddr{Departamento de An\'alisis Matem\'atico y Matem\'atica Aplicada, Universidad Complutense de Madrid, 28040 Madrid, Spain}

\email{adrialli@ucm.es}

\author{Kristian Seip} 
\address{Department of Mathematical Sciences, Norwegian University of Science and Technology (NTNU), 7491 Trondheim, Norway} 
\email{kristian.seip@ntnu.no}

\thanks{Research supported in part by Grants 275113 and 334466 of the Research Council of Norway. The work of Llinares was in particular funded by the former grant through the Alain Bensoussan Fellowship Programme of the European Research Consortium for Informatics and Mathematics. His research was also funded by the postdoctoral scholarship JCK22-0052 from the Kempe Foundations and partially supported by grant PID2019-106870GB-I00 from Ministerio de Ciencia e Innovaci\'{o}n. Seip's work was also supported in part by the Swedish Research Council under grant no. 2021-06594 while he was in residence at Institut Mittag-Leffler in Djursholm, Sweden during the spring semester 2024.}

\begin{abstract}
	Let $\mathfrak{p}_d(q)$ denote the critical exponent of the Riesz projection from $L^q(\mathbb{T}^d)$ to the Hardy space $H^p(\mathbb{T}^d)$, where $\mathbb{T}$ is the unit circle. We present the state-of-the-art on the conjecture that $\mathfrak{p}_1(q) = 4(1-1/q)$ for $1 \leq q \leq \infty$ and prove that it holds in the endpoint case $q = 1$. We then extend the conjecture to
	\[\mathfrak{p}_d(q) = 2+\cfrac{2}{d+\cfrac{2}{q-2}}\]
	for $d\geq1$ and $\frac{2d}{d+1} \leq q \leq \infty$ and establish that if the conjecture holds for $d=1$, then it also holds for $d=2$. When $d=2$, we verify that the conjecture holds in the endpoint case  $q = 4/3$. 
\end{abstract}


\maketitle
\section{Introduction}
This paper deals with what we will call the \emph{critical exponent conjecture} for the Riesz projection. This problem was first formulated in \cite{BOSZ18} along with three other conjectures about contractive inequalities. The latter problems have been settled affirmatively in recent papers by Kulikov \cite{Kulikov22} and the second named author  \cite{Llinares22}, and so the critical exponent conjecture is the sole survivor from  \cite{BOSZ18}. 

We begin by considering the problem in the classical setting of the unit circle $\mathbb T$. We say that $\varphi$ belongs to the Hardy space $H^p$ for $0<p<\infty$ if $\varphi$ is analytic in the unit disc $\mathbb{D}$ and 
\[ \| \varphi\|_{H^p}\coloneqq\sup_{0<r<1} \left(\int_{0}^{2\pi} |\varphi(re^{i\theta})|^p \frac{d\theta}{2\pi}\right)^{\frac{1}{p}}<\infty. \] 
We think of $H^p$ as a subspace of $L^p\coloneqq L^p(\mathbb T)$. Indeed, a function $\varphi$ in $H^p$ has a nontangential limit at almost every point in $\mathbb T$, also denoted by $\varphi$, and we have $\| \varphi \|_{H^p}=\| \varphi \|_{p}$, where $\| \varphi \|_p \coloneqq \| \varphi \|_{L^p}$. If $\varphi$ is a nontrivial function in $H^p$ for some $p > 0$, then we set
\[\|\varphi\|_0 \coloneq \exp\left(\int_0^{2\pi} \log|\varphi(e^{i\theta})|\,\frac{d\theta}{2\pi}\right),\]
so that $\|\varphi\|_p \to \|\varphi\|_0$ as $p \to 0^+$. The \emph{Riesz projection} $P_+$ is the orthogonal projection from $L^2$ to $H^2$. Since functions in $L^q$ for $q\ge 1$ can be represented by Fourier series, we may write
\[ P_+ \varphi (e^{i\theta})=\sum_{n=0}^{\infty} \widehat{\varphi}(n) e^{i n\theta} \]
and thus consider $P_+$ as an operator acting on $L^q$ for $q\ge 1$.
Its \emph{critical exponent} is defined for $1 \leq q \leq \infty$ as 
\[	\mathfrak{p}(q) \coloneq \sup \left\{ p\geq 0: \ \| P_+\psi \|_p \leq \| \psi  \|_q \, \text{ for all }\,  \psi \, \text{ in }\,  L^q\right\} \]
if the set on the right-hand side is nonempty, and otherwise we set by convention $\mathfrak{p}(q)=-1$. 
It is plain that $\mathfrak{p}(2)=2$ by orthogonality and that $q \mapsto \mathfrak{p}(q)$ is increasing by H\"older's inequality. Marzo and Seip \cite{MZ11} have established that $\mathfrak{p}(\infty)=4$, which by duality also yields $\mathfrak{p}(4/3)=1$. Our starting point reads as follows.

\begin{figure}
	\centering
	\begin{tikzpicture}
		\begin{axis}[
			axis equal image,
			xlabel= $q$, 
			ylabel = $\mathfrak{p}(q)$,
			axis lines = center,
			xmin = -0.01,
			xmax = 4,
			ymin = -0.01,
			ymax = 4,
			every axis x label/.style={
				at={(ticklabel* cs:1.025)},
				anchor=west,
			},
			every axis y label/.style={
				at={(ticklabel* cs:1.025)},
				anchor=south,
			},
			xtick={0.01,1,2,4},
			xticklabels={1,\sfrac{4}{3},2,$\infty$},
			ytick={1,2,4},
			yticklabels={1,2,4},
			axis line style={-}]
			
			
			\addplot [domain=0:1, samples=500, color=red, thick] {0};
			\addplot [domain=1:2, samples=500, color=red, thick] {2/(3-x)};
			\addplot [domain=2:4, samples=500, color=red, thick] {8/(6-x)};
			
			\addplot [domain=0:4, samples=500, color=blue, thick] {x};
						
		\end{axis} 
	\end{tikzpicture}
	\caption{The current {\color{blue} upper} and {\color{red} lower} bounds for $\mathfrak{p}(q)$. The upper bound is from \cite{BOSZ18} or Corollary~\ref{cor:upperbound} below. The first part of the lower bound is from Theorem~\ref{thm:L1}, while the other two are via \cite{MZ11} and Riesz--Thorin interpolation (see \eqref{eq:btwn2inf} and \eqref{eq:btwn12} below). The critical exponent conjecture asserts that the upper bound is correct.}
	\label{fig:p1q}
\end{figure}
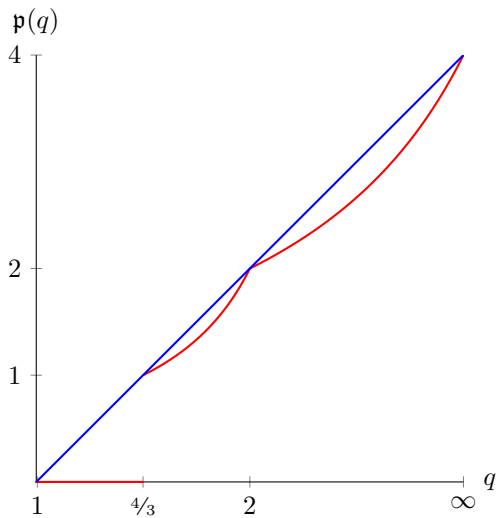

\begin{cec}
	$\mathfrak{p}(q) = 4(1-1/q)$ for $1 \leq q \leq \infty$.	
\end{cec}

Note that the endpoint case $q=1$ of the conjecture would follow from the case $1<q<2$ in the limit. It is known that $\mathfrak{p}(q) \leq 4(1-1/q)$, so the task at hand is to establish the converse inequality.

The purpose of the present note is twofold: Our first goal is to present the state-of-the-art on the critical exponent conjecture and explain how the conjecture fits within the framework of dual extremal problems on Hardy spaces. Our second goal is to extend the conjecture to the higher-dimensional setting.

We will begin by resolving the critical exponent conjecture in the endpoint case $q=1$. Recall that every nontrivial function $\varphi$ in $H^p$ enjoys the canonical factorization $\varphi = \Phi I$, where $\Phi$ is an outer function and $I$ is an inner function.

\begin{theorem} \label{thm:L1}
	If $\psi$ is in $L^1(\mathbb{T})$, then
	\begin{equation} \label{eq:explog} \exp\left(\int_0^{2\pi} \log|P_+ \psi(e^{i\theta})|\,\frac{d\theta}{2\pi}\right) \leq \|\psi\|_1.\end{equation}
	Equality occurs for a nontrivial function $\psi$ in $L^1$ if and only if $\overline{I} \psi \geq0$ almost everywhere, where $I$ denotes the inner factor of $P_+ \psi$.
\end{theorem}

In combination with \cite[Theorem~9]{BOSZ18} or Theorem~\ref{thm:RPK} below, we get the following.

\begin{corollary} \label{cor:L1}
	$\mathfrak{p}(1)=0$.
\end{corollary}

See Figure~\ref{fig:p1q} for a visualization of the current status on the critical exponent conjecture. 
 
There are of course many functions in $L^p$ whose Riesz projections coincide. When considering the critical exponent conjecture, it is natural to restrict to the function of minimal $L^p$ norm yielding a given Riesz projection. This means that we may invoke the theory of dual extremal problems in Hardy spaces, which in its final form is independently due to Havinson~\cite{Havinson1951} and to Rogosinski and Shapiro \cite{RS1953}.

In the case $p=1$, a function $\psi$ has minimal $L^1$ norm among all the functions with Riesz projection $P_+ \psi$ if and only if there is an inner function $I$ such that $\overline{I} \psi \geq 0$ almost everywhere (see e.g. \cite[Theorem~2.1]{AF2018}). We see that equality in \eqref{eq:explog} is attained if and only if this inner function happens to be the inner factor of $P_+ \psi$. A trivial example will be $\psi = CI$ for a constant $C\neq0$ and an inner function $I$. We will explain below how to construct more interesting examples of functions $\psi$ for which equality holds in \eqref{eq:explog}.

The main motivation for \cite{MZ11} and, to a somewhat lesser extent, behind the critical exponent conjecture \cite{BOSZ18} was to understand the Riesz projection in higher dimensions. Let $\mathbb{T}^d$ be the $d$-fold Cartesian product of the circle $\mathbb{T}$. Every $\psi$ in $L^q(\mathbb{T}^d)$ for $q\ge 1$ can be expanded as a Fourier series
\[\psi(z) \sim  \sum_{\alpha \in \mathbb{Z}^d} \widehat{\psi}(\alpha) \,z^\alpha. \]
We define $H^p(\mathbb{T}^d)$ for $p>0$ similarly as above (see \cite[Chapter~3]{Rudin}). For $p\ge 1$, $H^p(\mathbb{T}^d)$ is comprised of those functions in $L^p(\mathbb{T}^d)$ that satisfy \[\widehat{\psi}(\alpha)=0\]
for every $\alpha$ in $\mathbb{Z}^d \setminus \mathbb{N}_0^d$, where $\mathbb{N}_0=\{0,1,2,\ldots\}$. We define $\|\varphi\|_0$ as in the one-dimensional case and have again that $\|\varphi\|_p \to \|\varphi\|_0$ as $p \to 0^+$. The Riesz projection $P_+$ is the orthogonal projection from $L^2(\mathbb{T}^d)$ to $H^2(\mathbb{T}^d)$ and its \emph{critical exponent} is defined as before for $1\le q \le \infty$:
\[ \mathfrak{p}_d(q) \coloneqq \sup\left \{ p\geq 0 \,: \, \| P_+\psi \|_p \leq \| \psi  \|_q \, \text{ for every }\, \psi \, \text{ in } \, L^q(\mathbb T^d)\right\}, \]
with the same convention that $\mathfrak{p}(q)=-1$ should the set on the right-hand side be empty. The \emph{minimal admissible exponent} in dimension $d$, defined as
\begin{equation} \label{eq:qd}
	\mathfrak{q}(d) \coloneqq \inf \{q: \mathfrak{p}_d(q) \ge  0   \},
\end{equation}
is a quantity of particular interest in this context. Since $\| P_+ f\|_0\le \|f\|_q$ for every $q>\mathfrak{q}(d)$, it is clear that $\mathfrak{p}(\mathfrak{q}(d))=0$, so the infimum \eqref{eq:qd} is in fact attained.

A simple argument involving H\"older's inequality (see Lemma~\ref{lem:holder} below) shows that
\begin{equation} \label{eq:holder}
	\mathfrak{p}_{d+1}(q) \geq \mathfrak{p}_d(\mathfrak{p}_1(q)),
\end{equation}
so one way to extend the critical exponent conjecture to higher dimensions is to find the function of two variables $d$ and $q$ that agrees with the conjecture in the case $d=1$ and that attains equality in \eqref{eq:holder}. We see that
\[\mathfrak{a}_d(q) \coloneq 2+\cfrac{2}{d+\cfrac{2}{q-2}},\]
satisfies $\mathfrak{a}_1(q)=4(1-1/q)$ and the functional equation $\mathfrak{a}_{d_1+d_2}(q) = \mathfrak{a}_{d_1}(\mathfrak{a}_{d_2}(q))$, and it is therefore the desired function. We arrive at the following generalization.

\begin{figure}
	\centering
	\begin{tikzpicture}
		\begin{axis}[
			axis equal image,
			xlabel= $q$, 
			ylabel = $\mathfrak{p}_2(q)$,
			axis lines = center,
			xmin = -0.01,
			xmax = 4,
			ymin = -1.01,
			ymax = 3,
			every axis x label/.style={
				at={(ticklabel* cs:1.025)},
				anchor=west,
			},
			every axis y label/.style={
				at={(ticklabel* cs:1.025)},
				anchor=south,
			},
			xtick={1,3/2,2,4},
			xticklabels={\sfrac{4}{3},\sfrac{8}{5},2,$\infty$},
			ytick={-1,1,2,8/3,3},
			yticklabels={-1,1,2,\sfrac{8}{3},3},
			axis line style={-}]
			
			
			\addplot [domain=1:3/2, samples=500, color=red, thick] {0};
			\addplot [domain=3/2:2, samples=500, color=red, thick] {2/(5-2*x)};
			\addplot [domain=2:4, samples=500, color=red, thick] {16/(10-x)};
			
			\addplot [domain=0:1, samples = 500, color=blue, thick] {-1};
			\addplot [domain=1:4, samples=500, color=blue, thick] {4-4/x};
						
		\end{axis} 
	\end{tikzpicture}
	\caption{The current {\color{blue} upper} and {\color{red} lower} bounds for $\mathfrak{p}_2(q)$. The upper bounds are from Theorem~\ref{thm:d2}. The first part of the lower bound is from Corollary~\ref{cor:p243}, while the other two are obtained from the lower bounds in $d=1$ and Lemma~\ref{lem:holder}. The critical exponent conjecture asserts that the upper bound is correct.}
	\label{fig:p2q}
\end{figure}
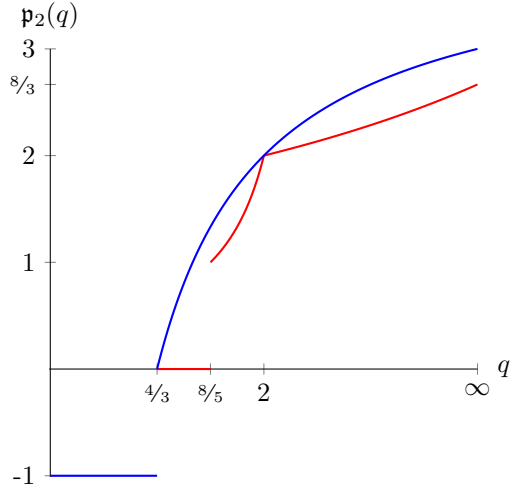

\begin{cec}[$d>1$] 
	\[ \mathfrak{p}_d(q) = \begin{cases} \mathfrak{a}_d(q), & \frac{2d}{d+1} \le q \le \infty; \\
	  -1, & 1\le q < \frac{2d}{d+1} . \end{cases} \]
\end{cec}
If we assume that the critical exponent conjecture holds for $d=1$ and every $1 \leq q \leq \infty$, then it follows from \eqref{eq:holder} and induction that
\begin{equation} \label{eq:upandaway}
	\mathfrak{p}_d(q) \geq \mathfrak{a}_d(q)
\end{equation}
for every $d\geq1$ and every $\frac{2d}{d+1} \leq q \leq \infty$. If we could find examples in higher dimensions allowing us to reverse the inequality in \eqref{eq:upandaway}, then the critical exponent conjecture in dimension $d$ would follow from that in the case $d=1$.  We have found such examples when $d=2$ and thus obtained the following result.

\begin{theorem} \label{thm:d2}
If $\frac{4}{3} \leq q \leq \infty$, then $\mathfrak{p}_2(q) \le \mathfrak{a}_2 (q)$. If $0 < q < \frac{4}{3}$, then $\mathfrak{p}_2 (q) = -1$.
\end{theorem}

In view of \eqref{eq:upandaway}, this theorem implies the following.

\begin{corollary}
	If the critical exponent conjecture for the Riesz projection holds for $d=1$, then it also holds for $d=2$.
\end{corollary}

Since $\mathfrak{p}_1\big(\frac{4}{3}\big)=1$ and $\mathfrak{p}_1(1)=0$, we are also able to compute $\mathfrak{p}_2(4/3)$.

\begin{corollary} \label{cor:p243}
	$\mathfrak{p}_2 (4/3) = 0$.
\end{corollary}

See Figure~\ref{fig:p2q} for a visualization of the current status on the critical exponent conjecture in dimension $d=2$.

The problem of constructing suitable examples appears to be much harder when $d\geq 3$. Indeed, it may well be that the cases $d=1$ and $d=2$ are exceptional and that new phenomena occur for $d\ge 3$, in a similar way as was seen in a somewhat related context in \cite[Theorem~1.2]{BOS2021}. This state of affairs means that our conjecture in higher dimensions is not as well founded as in the low dimensions $d=1$ and $d=2$.

However, we do have examples exhibiting the right asymptotic behavior when $d\to \infty$. Konyagin, Queff\'{e}lec, Saksman, and Seip have recently established that $\mathfrak{p}_d(q) \to 2$ as $d \to \infty$ for every fixed $2 \leq q \leq \infty$ (see \cite[Corollary~2]{KQSS22}), but their argument does not provide an effective upper bound for $\mathfrak{p}_d(q)$ to compare with the critical exponent conjecture. In analogy with this result, it is natural to expect that $\mathfrak{q}(d) \to 2$ as $d \to \infty$ (see Section~\ref{sec:Td} for an elaboration).

\subsection*{Organization} In Section~\ref{sec:sota}, we summarize the state-of-the-art on the critical exponent conjecture and discuss its relationship with dual extremal problem on Hardy spaces. Section~\ref{sec:L1} contains new evidence in favour of the critical exponent conjecture, including the proof of Theorem~\ref{thm:L1} and a discussion about its associated extremal functions. Section~\ref{sec:Td} concerns the critical exponent conjecture in higher dimensions and presents in particular the proof of Theorem~\ref{thm:d2}.

\section{The state-of-the-art and dual extremal problems} \label{sec:sota}
It is well-known that the Riesz projection extends to a bounded linear operator from $L^q$ to $H^q$ for $1<q<\infty$. In fact, Hollenbeck and Verbitsky \cite{HB00} proved that
\begin{equation} \label{eq:HV}
	\|P_+\|_{L^q \to H^q} = \frac{1}{\sin{\pi/q}}
\end{equation}
Although there are $\psi$ in $L^\infty$ such that $P_+ \psi$ is not in $H^\infty$, it follows from the above and H\"older's inequality that $P_+\psi$ is in $H^p$ for every $p < \infty$. We will let $H^p_0$ denote the subspace of $H^p$ consisting of functions that vanish at the origin. If $1<p<\infty$, then every $\psi$ in $L^q$ can be expressed as $\psi = \varphi + \overline{\varphi_0}$ for $\varphi$ in $H^q$ and $\varphi_0$ in $H^q_0$.

We begin with the proof of the estimate $\mathfrak{p}(\infty) \geq 4$. Let $P_-$ denote the orthogonal projection from $L^2$ to $\overline{H^2_0}$, so that $\psi = P_+ \psi + P_- \psi$. The following argument is due to Marzo and Seip \cite{MZ11} although the second statement in the result is new.

\begin{theorem} \label{thm:MZ}
	If $\psi$ is in $L^\infty$, then
	\[\|P_+ \psi\|_4 \leq \|\psi\|_\infty.\]
	Equality holds if and only if $\psi = CI$ for a constant $C$ and an inner function $I$.
\end{theorem}

\begin{proof}
	Since $\psi$ is in $L^\infty$, it is plain that both $(P_+ \psi)^2$ and $(P_- \psi)^2$ are in $L^2$ and that $(P_+ \psi)^2 \perp (P_-\psi)^2$. Consequently,
	\begin{align}
		\|P_+ \psi\|_4^4 = \|(P_+ \psi)^2\|_2^2 &\leq \|(P_+ \psi)^2-(P_- \psi)^2\|_2^2 \label{eq:MZ1} \\
		&=\|\psi (P_+ \psi - P_- \psi)\|_2^2 \leq \|\psi\|_\infty^2 \|\psi\|_2^2 \leq \|\psi\|_\infty^4. \label{eq:MZ2}
	\end{align}
	The estimate in \eqref{eq:MZ1} is saturated if and only if $P_- \psi = 0$. Consequently, $\psi$ is in $H^\infty$. The inequality in \eqref{eq:MZ2} is saturated if and only if $\psi$ has constant modulus.
\end{proof}

Since the Riesz projection is self-adjoint, we get the following result from H\"older's inequality.

\begin{corollary} \label{cor:43}
	If $\psi$ is in $L^{4/3}$, then
	\[\|P_+ \psi\|_1 \leq \|\psi\|_{4/3}.\]
	Equality holds if and only if $\psi = CI$ for a constant $C$ and an inner function $I$.
\end{corollary}

\begin{proof}
	We write $\varphi = P_+ \psi$ and assume that $\varphi = I \Phi$ is nontrivial. We set $g = \varphi/|\varphi|$ and compute 
	\[\|P_+ \psi\|_1 = \langle P_+ \psi, g \rangle = \langle \psi, P_+ g \rangle \leq \|\psi\|_{4/3} \|P_+ g\|_4 \leq \|\psi\|_{4/3} \|g\|_\infty = \|\psi\|_{4/3}.\]
	In the final inequality we used Theorem~\ref{thm:MZ}, which means that equality can hold throughout if and only if $g = C I$. Since $\|g\|_\infty=1$ we have $g=I$ up to a unimodular constant. By the definition of $g$, this means that $\Phi = C \neq 0$.
\end{proof}

As remarked in \cite{MZ11}, if we use Riesz--Thorin interpolation between $\mathfrak{p}(\infty) = 4$ and $\mathfrak{p}(2)=2$, we obtain 
\begin{equation} \label{eq:btwn2inf}
	\mathfrak{p}(q) \geq \frac{4q}{q+2}
\end{equation}
when $2 \leq q \leq \infty$. By duality, we then also get 
\begin{equation} \label{eq:btwn12}
	\mathfrak{p}(q) \geq \frac{2q}{4-q}
\end{equation}
when $\frac{4}{3} \leq q \leq 2$. These lower bounds are strictly smaller than the conjectured value $\mathfrak{p}(q) = 4/q^\ast$ in the whole range $\frac{4}{3}<q<\infty$, except in the trivial case $q=2$. Since no known interpolation result applies when $0<p<1$, we then only have the bound $\mathfrak{p}(q)\geq 0$, which is a trivial consequence of Theorem~\ref{thm:L1}. See Figure~\ref{fig:p1q} above.

We refer to \cite{BCOS2023,Kulikov22,Llinares22} for examples of other problems of this kind, with the classical tools of interpolation theory being either inadequate or unavailable. In particular, in \cite{Kulikov22,Llinares22}, the challenge was to break the convexity bound obtained by Riesz--Thorin interpolation, and this is the main problem in our context as well when $\frac{4}{3}<q<\infty$.

As mentioned in the introduction, when we study the critical exponent conjecture for the Riesz projection, we are only interested in the functions $\psi$ in $L^q$ of minimal norm among all functions with Riesz projection $P_+ \psi$. This leads us to the theory of dual extremal problems on Hardy spaces. A complete account of the theory can be found in Chapter 7 and Chapter 8 in Duren \cite{Duren70}.

The starting point is a nontrivial function $\varphi$ in $H^q$ for $1 \leq q < \infty$. The \emph{dual extremal problems} generated by $\varphi$ are
\[\sup_{f \in H^{q^\ast}} \frac{|\langle f, \varphi \rangle|}{\|f\|_{q^\ast}} \qquad \text{and} \qquad \inf_{\varphi_0 \in H^q_0} \|\varphi + \overline{\varphi_0}\|_q.\]
The expression on the left-hand side is nothing more than the norm of the bounded linear functional on $H^{q^\ast}$ generated by $\varphi$, while the expression on the right-hand side is trivially an upper bound for the norm of said linear functional. 

The salient points of the theory are as follows: The supremum is equal to the infimum and both are attained. In fact, there is a unique $\varphi_0$ in $H^q_0$ attaining the infimum. The corresponding function 
\[\psi = \varphi+\overline{\varphi_0}\]
is called the \emph{extremal kernel}, while $\varphi$ is called the \emph{natural kernel}. It is plain that there is no unique function attaining the supremum, but if we require that $\langle f, \varphi \rangle > 0$, then $f$ is unique up to multiplication by positive constants. Consequently, for such $f$ we deduce from the fact that $\varphi = P_+ \psi$ that
\begin{equation} \label{eq:holderme}
	\|\psi\|_q \|f\|_{q^\ast} = \langle f, \psi\rangle.
\end{equation}
Note that this equality appears with $q=4/3$ in the proof of Corollary~\ref{cor:43}.

For $1 < p < \infty$, consider next the nonlinear operator
\[\mathbf{N}_p f(z) \coloneq \begin{cases}
	|f(z)|^{p-2} f(z), & \text{if } f(z) \neq 0; \\
	0, & \text{if } f(z) = 0,
\end{cases}\]
which describes the case of equality in H\"older's inequality and is of relevance to $L^p$ orthogonality. For more see \cite{HSVZ18} and \cite[Chapter~5]{Shapiro71}. Returning to \eqref{eq:holderme}, we see that there is a unique $f$ attaining the supremum in the dual extremal problems such that the extremal kernel $\psi$ satisfies $\psi = \mathbf{N}_{q^\ast} f$. We find it convenient to call this function $f$ the \emph{extremal function}. The following diagram illustrates the one-to-one correspondences between the natural kernel $\varphi$ in $H^q$, the extremal function $f$ in $H^{q^\ast}$, and the extremal kernel $\psi$ in $L^q$.
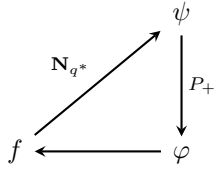
\[\begin{tikzcd}
	& & \psi \arrow[dd,"P_+"] \\ & & \\ 
	f \arrow[uurr,"\mathbf{N}_{q^\ast}"] & & \arrow[ll] \varphi
		\end{tikzcd}\]
This means that when we consider the critical exponent conjecture for $1 < q < \infty$, we can restrict our attention to $\psi$ of the form $\psi = \mathbf{N}_{q^\ast} f$ for $f$ in $H^{q^\ast}$.

In the case $q=1$ there is a unique inner function $I$ such that $\|\psi\|_1 = \langle I, \varphi \rangle$ and, consequently, such that $\overline{I} \psi \geq 0$ almost everywhere. We say that this inner function is the \emph{extremal function}. The key difference between the case $q=1$ and $1<q<\infty$ is that the same inner function $I$ may be the extremal function for several different extremal kernels $\psi$, so the arrow from $f$ to $\psi$ in the diagram breaks down.

In the case $q=\infty$, the situation described above may break down in several ways. We will not go into details and refer again to Duren \cite{Duren70} for a complete account. However, \emph{if} there is a function $f$ in $H^1$ that attains the supremum, then there is a unique extremal kernel of the form $\psi = C \mathbf{N}_1 f$ where $C$ is equal to the supremum so that $\|\psi\|_\infty = \|\varphi\|_{(H^1)^\ast}$.

\section{On extremal functions and the endpoint case \texorpdfstring{$q=1$}{q=1}} \label{sec:L1}
We begin by explaining how $\|\varphi\|_0$ can be computed for a nontrivial function in $H^p$. As mentioned above, every nontrivial function $\varphi$ in $H^p$ enjoys the inner-outer factorization $\varphi = I \Phi$. Since $|\varphi| = |\Phi|$ almost everywhere on $\mathbb{T}$ and since $\log|\Phi|$ is harmonic in $\mathbb{D}$, we have
\begin{equation} \label{eq:explogouter}
	\|\varphi\|_0 = \exp\left(\int_0^{2\pi} \log|\Phi(e^{i\theta})|\,\frac{d\theta}{2\pi}\right) = \exp\left(\log|\Phi(0)|\right) = |\Phi(0)|.
\end{equation}

We will next consider the dual extremal problems in a specific case and see how it relates to the critical exponent conjecture. Let $w$ be a point in the unit disc and let $k_w$ denote the natural kernel
\[k_w(z) \coloneq \frac{1}{1-\overline{w}z}.\]
If $f$ is a function in $H^{q^\ast}$ for $1\leq q^{\ast} \leq \infty$, then $f(w) = \langle f, k_w \rangle$. Recall next the well-known sharp pointwise estimate
\begin{equation} \label{eq:pest}
	|f(w)| \leq \left(1-|w|^2\right)^{-\frac{1}{q^\ast}} \|f\|_{q^{\ast}}.
\end{equation}
In the case $q^\ast=\infty$, the estimate \eqref{eq:pest} is simply the maximum modulus principle and, hence, the estimate is saturated if and only $f$ is identically equal to a constant. The standard proof of \eqref{eq:pest} for $1 \leq q^\ast < \infty$ goes via the inner-outer factorization. It is first established for $q^{\ast}=2$ using the Cauchy--Schwarz inequality and orthogonality, then extended to the general case using the inner-outer factorization. It follows from this argument that the estimate is saturated if and only if
\begin{equation} \label{eq:extremalrpk}
	f(z) = C \left(1-\overline{w}z\right)^{-\frac{2}{q^\ast}}
\end{equation}
for a constant $C\neq0$. By the theory of dual extremal problems, which in this case applies in the whole range $1 \leq q^\ast \leq \infty$, this is equivalent to the assertion that the extremal kernel $\psi$ in $L^q$ for the natural kernel $k_w$ satisfies 
\begin{equation} \label{eq:rpkek}
	\|\psi\|_q = \left(1-|w|^2\right)^{-\frac{1}{q^\ast}}
\end{equation}
for $1 \leq q \leq \infty$. In the case $1<q\leq\infty$, the extremal kernel can be computed using $\mathbf{N}_{q^\ast}$ and \eqref{eq:extremalrpk}. In the case $q=1$, all we can say is that the extremal kernel $\psi$ is a nonnegative function satisfying $\|\psi\|_1 = 1$. A moment's thought reveals that in this case $\psi$ is simply the Poisson kernel at $w$. However, this information about $\psi$ is irrelevant for what we will do next since we will only use \eqref{eq:rpkek}.

\begin{theorem} \label{thm:RPK}
	Suppose that $1\leq q \leq \infty$. The inequality
	\[\|k_w\|_p \leq \|\psi\|_q\]
	holds for all $\psi$ with $P_+ \psi = k_w$ and for all $w$ in the unit disc if and only if $p \leq 4/q^\ast$. If $\psi$ has minimal norm and $P_+ \psi = k_w$, then we have $\| \psi \|_q = \| k_w \|_{4/q^\ast}$ for all $w$ when $q=1$ or $q=2$ and only for $w=0$ otherwise.
\end{theorem}

\begin{proof}
	The case $q=1$ and $p=0$ is trivial in view of \eqref{eq:explogouter} since $k_w$ is outer and $k_w(0)=1$ for every $w$ in $\mathbb{D}$. We will therefore assume in what follows that $1<q \leq\infty$. Let $\psi$ be the extremal kernel for the natural kernel $k_w$. We set $|w|^2 = r$ and use \eqref{eq:rpkek} and the binomial series to compute
	\[\|\psi\|_q^p = \sum_{n=0}^\infty \binom{n-1+p/q^\ast}{n} r^n.\]
	Writing $\|k_w\|_p^p = \|k_w^{p/2}\|_2^2$ we find that
	\[\|k_w\|_p^p = \sum_{n=0}^\infty \binom{n-1+p/2}{n}^2 r^n.\]
	We can obtain a necessary condition by letting $r \to 0^+$ and checking only the coefficient $n=1$. This yields
	\[\frac{p}{q^\ast} \geq \left(\frac{p}{2}\right)^2 \qquad \iff \qquad \frac{4}{q^\ast} \geq p.\]
	To see that this condition on $p$ is also sufficient need to show that 
	\[\binom{n-1+(p/2)^2}{n} \geq \binom{n-1+p/2}{n}^2.\]
	This inequality follows from the product formula for binomial coefficients and the elementary estimate
	\[\frac{j-1+(p/2)^2}{j} \geq \frac{(j-1+p/2)^2}{j^2},\]
	that holds for $j=1,2,3,\ldots$ as is easily verified. The final inequality is strict for $j>1$ unless $p=2$.
\end{proof}

The ``only if'' part of the previous result was mentioned on \cite[p.~495]{Brevig19}, while the ``if'' part is new. The ``only if'' part of Theorem~\ref{thm:RPK} has the following consequence, which was shown in \cite[Theorem~9]{BOSZ18} with a different proof.

\begin{corollary} \label{cor:upperbound}
	If $1\le q \leq \infty$, then $\mathfrak{p}(q) \leq 4/q^\ast$.
\end{corollary}

In combination with Theorem~\ref{thm:MZ} and Corollary~\ref{cor:43}, we now recover that $\mathfrak{p}(4/3)=1$ and $\mathfrak{p}(\infty)=4$.

We next turn to the proof of Theorem~\ref{thm:L1} and begin with a small technical point. Although there are $\psi$ in $L^1$ such that $P_+ \psi$ is not in $H^1$, a classical theorem (see e.g. \cite[Chapter VII.2]{Zygmund2002}) asserts that there is a constant $C>0$ such that
\[\|P_+ \psi\|_p \leq \frac{C}{1-p} \|\psi\|_1\]
for $0<p<1$ and every $\psi$ in $L^1$. In particular, this means that $P_+ \psi$ enjoys the inner-outer factorization.

\begin{proof}[Proof of Theorem~\ref{thm:L1}]
	Let $\varphi = P_+ \psi$ have inner-outer factorization $\varphi = I \Phi$, so that $\|\varphi\|_0 = |\Phi(0)|$ by \eqref{eq:explogouter}. Writing $\psi = \overline{\varphi_0} + I \Phi$ for $\varphi_0$ in $H^p_0$ for $p<1$, we see that $\Phi = P_+ (\overline{I}\psi)$. Hence $\Phi(0)$ is the $0$th Fourier coefficient of $\overline{I}\psi$, which shows that
	\[|\Phi(0)| \leq \|\overline{I}\psi\|_1 = \|\psi\|_1.\]
	We have $|\Phi(0)|=\|\psi\|_1$ if and only if $\overline{I}\psi \geq0$ almost everywhere.
\end{proof}

Let us compare the case of equality in Theorem~\ref{thm:L1} with that of Theorem~\ref{thm:MZ} and Corollary~\ref{cor:43}. In the latter case, equality was attained only when $\psi = CI$ for a constant $C$ and an inner function $I$. Such functions yield trivially equality in Theorem~\ref{thm:L1} as well, but the following example shows that  equality can occur for other functions in this case. Let $\Psi$ be a nonnegative function in $L^1$ that can be written as
\[\Psi = \overline{J} + \Phi,\]
where $J$ is an inner function that vanishes at the origin and $\Phi$ is an outer function. Let $I$ be a proper divisor of $J$ such that $J/I$ vanishes at the origin. Then $\|P_+\psi\|_0 = \|\psi\|_1$ when $\psi = I \Psi$. Perhaps the simplest example is obtained from $J(z)=z^2$, $\Phi(z) = 2+z^2$, and $I(z)=z$, so that 
\[\psi(e^{i\theta}) = e^{-i\theta} + 2e^{i\theta} + e^{3i\theta} = 2e^{i\theta} (1+\cos{2\theta}).\]
Another example can be found in Theorem~\ref{thm:RPK}. If $\psi$ is the Poisson kernel at $w$ in $\mathbb{D}$, the $P_+ \psi = k_w$ and $1 = \|\psi\|_1 = \|k_w\|_0$. In this case, the inner function $I$ is identically equal to $1$ and $\psi \geq 0$.

We believe that the endpoint case $q=1$ is an outlier in this context and that equality in the general case should be the same as in the cases $q=\infty$ and $q=4/3$. 

\begin{conjecture}
	Suppose that $1<q<2$ or $2<q\leq \infty$. If there is a nontrivial function $\psi$ in $L^q$ such that
	\[\|P_+ \psi \|_{\mathfrak{p}(q)} = \|\psi\|_q,\]
	then there is a constant $C\neq0$ and an inner function $I$ such that $\psi = CI$. 
\end{conjecture}

Note that his conjecture is in line with Theorem~\ref{thm:RPK} which exhibits a striking contrast between the case $q=1$ and the cases $1<q<2$ and $2< q \leq \infty$.

\section{Critical exponents in higher dimensions} \label{sec:Td}
We begin by noting that $P_+$ enjoys the same $L^p$ boundedness properties on $\mathbb T^d$ as in the one-dimensional case. Indeed, when $1<q<\infty$, we get from \eqref{eq:HV} that
\[\|P_+ \|_{L^q(\mathbb{T}^d) \to H^q(\mathbb{T}^d)} = \frac{1}{(\sin{\pi/q})^d}.\]
Moreover, taking into account that $P_+$ is a singular integral operator, we may use the Calder\'on--Zygmund decomposition to show that $P_{+}$ is of weak-type $(1,1)$. (See for instance \cite[Section~II.2.4]{Stein70} for a detailed exposition of this fact.) In particular, we have that $P_+$ is bounded from $L^1 (\mathbb{T}^d)$ to $H^p(\mathbb{T}^d)$ for all $0 < p < 1$.

In preparation for the proof of Theorem~\ref{thm:d2}, recall that a function $f$ in $L^1(\mathbb{T}^2)$ is called \emph{$2$-homogeneous} if
\begin{equation} \label{eq:2hom1}
	f(e^{i\theta}z_1,e^{i\theta}z_2) = e^{2i\theta} f(z_1,z_2)
\end{equation}
for almost every $z$ on $\mathbb{T}^2$ and almost every real number $\theta$. It is not difficult to check that $f$ is $2$-homogeneous if and only if
\begin{equation} \label{eq:2hom2}
	\widehat{f}(\alpha) = 0
\end{equation}
for every $\alpha=(\alpha_1,\alpha_2)$ in $\mathbb{Z}^2$ such that $\alpha_1+\alpha_2 \neq 2$. This means that the $2$-homogeneous functions in $H^1(\mathbb{T}^2)$ are simply the polynomials
\[f(z) = az_1^2 + bz_1z_2 + cz_2^2.\]
It follows from \eqref{eq:2hom1} that if $f$ is $2$-homogeneous, then so is $\mathbf{N}_{q^\ast} f$. A consequence of \eqref{eq:2hom2} is that if $\psi$ is $2$-homogeneous, then so is $P_+ \psi$. As a final preparation for the proof of Theorem~\ref{thm:d2}, it will be convenient to use that $\mathfrak{a}_2(q) = 4-q^\ast$.

\begin{proof}[Proof of Theorem~\ref{thm:d2}]
	Suppose that $4/3<q\leq \infty$. Fix $0<\varepsilon<1/4$ and consider the $2$-homogeneous polynomial $f\coloneqq f_1 + \varepsilon f_2$, where
	\[f_1(z) \coloneqq z_1 z_2 \qquad \text{and} \qquad f_2(z) \coloneqq z_1^2-z_2^2.\]
	Guided by the above discussion about dual extremal problems, we set $\psi \coloneqq \mathbf{N}_{q^\ast}f$ and $\varphi \coloneqq P_+ \psi$, where as before $\mathbf{N}_{q^\ast}f\coloneqq |f|^{q^\ast-2} f$. Our strategy is to show that if
	\[\|\varphi\|_p \leq \|\psi\|_q\]
	holds for every $0<\varepsilon<1/4$, then $p \leq 4-q^\ast$. This would show that $\mathfrak{p}_2(q) \leq 4-q^\ast$.  The key idea idea underpinning our choice of $f$ is that $\overline{f_1} f_2$ is purely imaginary. Using this and that $f_1$ is unimodular, we find that
	\begin{equation} \label{eq:imag}
		|f(z)|^2 = 1+\varepsilon^2 |f_2(z)|^2.
	\end{equation}
	If $j\geq0$ is an integer, then using the binomial theorem and orthogonality, we find that
	\begin{equation} \label{eq:2jj}
		\|f_2\|_{2j}^{2j} = \|f_2^j\|_2^2 = \sum_{k=0}^j \binom{j}{k}^2 = \binom{2j}{j}.
	\end{equation}
	If $1<q<\infty$, then we combine \eqref{eq:imag} and \eqref{eq:2jj} to compute
	\[\|\psi\|_q^q = \|\mathbf{N}_{q^\ast} f\|_q^q = \|f\|_{q^\ast}^{q^\ast} = \sum_{j=0}^\infty \binom{q^\ast/2}{j} \binom{2j}{j} \varepsilon^{2j}\]
	to deduce that
	\begin{equation} \label{eq:fq}
		\|\psi\|_q = 1 + (q^\ast-1)\varepsilon^2 + \frac{(q^\ast-1)(3q^\ast-8)}{4}\varepsilon^4 + O(\varepsilon^6).
	\end{equation}
	The equation \eqref{eq:fq} also holds for $q=\infty$, since $\|\psi\|_\infty = 1$ and $\infty^\ast = 1$. Since $f$ is $2$-homogeneous, it follows that $\psi$ and, consequently, $\varphi$ are $2$-homogeneous. Hence we only need to compute $\widehat{\psi}(2,0)$, $\widehat{\psi}(1,1)$, and $\widehat{\psi}(0,2)$ to find $\varphi$. We proceed by using \eqref{eq:imag} to the effect that
	\begin{equation} \label{eq:varphiqast}
		|f(z)|^{q^\ast-2} = \sum_{j=0}^\infty \binom{q^\ast/2-1}{j} \varepsilon^{2j} |f_2(z)|^{2j}.
	\end{equation}
	Making use of the symmetries $f_1(z_2,z_1) = f_1(z_1,z_2)$ and $f_2(z_2,z_1) = - f_2(z_1,z_2)$, we deduce from \eqref{eq:varphiqast} and \eqref{eq:2jj} that
	\[\widehat{\psi}(1,1) = \langle \psi, f_1 \rangle = \langle |f|^{q^\ast-2}, 1 \rangle = \sum_{j=0}^\infty \binom{q^\ast/2-1}{j} \binom{2j}{j}\varepsilon^{2j}.\]
	The symmetries also show that $\widehat{\psi}(2,0) = - \widehat{\psi}(0,2)$, which when used in conjunction with \eqref{eq:varphiqast}, the symmetries again, and \eqref{eq:2jj} allows us to compute
	\[\widehat{\psi}(2,0) = \frac{\langle \psi, f_2 \rangle}{\|f_2\|_2^2} = \frac{\langle |f|^{q^\ast-2}, \varepsilon |f_2|^2 \rangle}{\|f_2\|_2^2}  = \varepsilon \sum_{j=0}^\infty \binom{q^\ast/2-1}{j} \binom{2j+1}{j+1} \varepsilon^{2j}.\]
	It follows that $\varphi = P_+\psi = a f_1 +  \varepsilon b f_2$, for
	\begin{equation} \label{eq:ab}
		a = \sum_{j=0}^\infty \binom{q^\ast/2-1}{j} \binom{2j}{j}\varepsilon^{2j} \qquad \text{and} \qquad b=\sum_{j=0}^\infty \binom{q^\ast/2-1}{j} \binom{2j+1}{j+1} \varepsilon^{2j}.
	\end{equation}
	Since $0<\varepsilon<1/4$, we can compute using \eqref{eq:imag} as above to find that
	\begin{equation} \label{eq:Pf}
		\|\varphi\|_p = a \Bigg(\sum_{j=0}^\infty \binom{p/2}{j} \binom{2j}{j} \left(\frac{b}{a}\right)^{2j} \varepsilon^{2j}\Bigg)^{\frac{1}{p}}.
	\end{equation}
	We extract from \eqref{eq:ab} that
	\begin{align*}
		a &= 1+ (q^\ast-2) \varepsilon^2 + \frac{3}{4}(q^\ast-2)(q^\ast-4) \varepsilon^4+O(\varepsilon^6), \\
		\left(\frac{b}{a}\right)^2 &= 1+ (q^\ast-2) \varepsilon^2 + O(\varepsilon^4),\\
		\left(\frac{b}{a}\right)^4 &= 1+ O(\varepsilon^2),
	\end{align*}
	which when inserted into \eqref{eq:Pf} shows that
	\begin{equation} \label{eq:Pfeps}
		\|\varphi\|_p = 1 +(q^\ast-1) \varepsilon^2 + \frac{p+3(q^\ast)^2 -10q^\ast+4}{4}\varepsilon^4 + O(\varepsilon^6).
	\end{equation}
	As $\varepsilon \to 0^+$, we find from \eqref{eq:fq} and \eqref{eq:Pfeps} that $\|\varphi\|_p \leq \|\psi\|_q$ holds if and only if
	\[\frac{p+3(q^\ast)^2 -10q^\ast+4}{4} \leq \frac{(q^\ast-1)(3q^\ast-8)}{4} \qquad \iff \qquad p \leq 4 - q^\ast.\]
	The same analysis shows that $\| \varphi \|_p > \| \psi \|_{4/3}$ for any $p>0$ if $\varepsilon>0$ is sufficiently small, which means that $\mathfrak{p}_2 (4/3) \leq 0$. It remains to show that $\mathfrak{p}_2 (q) = -1$ when $1 < q < \frac{4}{3}$. Using the Taylor expansion of $\log(1 + x)$ and arguing as above, we get
	\begin{eqnarray*}
		\frac{\| \varphi \|_0}{a} = \exp \left( \frac{1}{2} \sum_{j = 1}^\infty \frac{(-1)^{j+1}}{j} \binom{2j}{j} \left(\frac{b}{a} \right)^{2j} \varepsilon^{2j} \right) = 1 + \left( \dfrac{b}{a} \right)^2 \varepsilon^2 - \left( \dfrac{b}{a} \right)^4 \varepsilon^4 + O(\varepsilon^6).
	\end{eqnarray*}
	This shows that \eqref{eq:Pfeps} holds for $p=0$, which allows us to infer that $\| \varphi \|_0 > \| \psi \|_q$ for any $1 < q < \frac{4}{3}$ if $\varepsilon>0$ is sufficiently small.
\end{proof}


The following result (which is referred to in \eqref{eq:holder} above) can be extracted from the the proof of \cite[Theorem~4]{MZ11}. We record it here for posterity and sketch the proof.

\begin{lemma} \label{lem:holder}
	Fix $d\geq1$ and $\frac{4}{3} \leq q\leq \infty$. Then 
	\[\mathfrak{p}_{d+1}(q) \geq  \mathfrak{p}_d(\mathfrak{p}_1(q)).\]
\end{lemma}

\begin{proof}
	If $\mathfrak{p}_d(\mathfrak{p}_1(q))=-1$, then there is nothing to prove, so we may assume that $\mathfrak{p}_d(\mathfrak{p}_1(q))\ge 0$. Let $P_+^{(d)}$ denote the Riesz projection with respect to the variables $z_1,z_2,\ldots,z_d$ and $P_+^{(1)}$ the Riesz projection with respect to the variable $z_{d+1}$. Then $P_+^{(d+1)}= P_+^{(d)} P_+^{(1)}$.  We now set $r = \mathfrak{p}_1(q)$ and $p = \mathfrak{p}_d(r)$. When $p>0$, we obtain the asserted result from the definition of $\mathfrak{p}_d(r)$, H\"older's inequality, the definition of $\mathfrak{p}_1(q)$, and then H\"older's inequality again. When $p=0$, we use Jensen's inequality after employing the definition of $\mathfrak{p}_d(r)$. Then we repeat the three remaining steps of the argument in the preceding case $p>0$.
\end{proof}

We will now conclude with some remarks and observations about the minimal admissible exponent $\mathfrak{q}(d)$ and the behavior of $\mathfrak{p}_d(q)$ and $\mathfrak{q}(d)$ when $d\to \infty$. Note that we have already proved that $\mathfrak{q}(1)=1$ and $\mathfrak{q}(2)=\frac{4}{3}$. Using the lower bound \eqref{eq:btwn12} in $d=1$ obtained from Riesz--Thorin interpolation and Lemma~\ref{lem:holder}, we may prove the following general result.
\begin{theorem}
We have $\frac{4}{3} \leq \mathfrak{q}(d)\leq  \frac{2}{1+2^{1-d}}$ for all $d\ge 2$.
\end{theorem}
\begin{proof} The lower bound follows from the fact that $d\mapsto \mathfrak{p}_d(q)$ is a decreasing function.  To establish the upper bound, we start by observing that \eqref{eq:btwn12} implies that
\[\mathfrak{p}_1\left(\frac{2}{2^{-d}+1}\right)\geq \frac{2}{2^{1-d}+1}.\]
Hence, invoking Lemma~\ref{lem:holder}, we find that 
\[\mathfrak{p}_d\left(\frac{2}{2^{1-d}+1}\right) \geq \mathfrak{p}_{d-1}\left(\frac{2}{2^{2-d}+1}\right). \]
Iterating this argument $d-2$ times, we deduce that
\[\mathfrak{p}_d\left(\frac{2}{2^{1-d}+1}\right) \geq \mathfrak{p}_{2}\left(\frac{4}{3}\right)=0, \]
which yields the desired upper bound.
\end{proof}
The critical exponent conjecture claims that $\mathfrak{q}(d)=\frac{2d}{d+1}$, and it also anticipates the exact value of $q\mapsto \mathfrak{p}_d(q)$ on the interval $[\mathfrak{q}(d),\infty)$. The following is a weaker conjecture that might be more tangible.

\begin{weakcec}
	The function $q\mapsto \mathfrak{p}_d(q)$ is continuous and strictly increasing on $[\mathfrak{q}(d), \infty)$. Moreover, $\mathfrak{q}(d) \to 2$ as $d \to \infty$.
\end{weakcec}
We note that $q\mapsto \mathfrak{p}_d(q)$ is indeed continuous when $\mathfrak{p}_d(q)\ge 1$. This is a consequence of the Riesz--Thorin interpolation theorem and monotonicity of $q\mapsto \mathfrak{p}_d(q)$. The following is the closest we are able to get in our attempt to  verify the second part of the conjecture.
\begin{theorem} \label{thm:qto2}
	We have $\lim_{d\to \infty} \mathfrak{p}_d(q)\leq 0$ for every $q<2$.
\end{theorem}
The proof of this theorem is essentially identical to the proof of \cite[Thm. 2.1]{KQSS22}. Of crucial importance is again the spherical Dirichlet kernel
\[ D_{R,d}(z)\coloneqq \sum_{\substack{\alpha\in \mathbb Z^d \\ \|\alpha\|\leq R}} z^\alpha, \]
where $\|\alpha \|^2\coloneqq \alpha_1^2+\cdots + \alpha_d^2$. We need the following estimate.
\begin{lemma}\label{lem:smallp}
There exists an absolute constant $c$, $0<c<1$, such that
\[ \| D_{R,d} \|_p \ge c^{\frac{1}{p}} R^{\frac{d-1}{2}} \]
when $d>1$ and $0<p\le 1$.
\end{lemma}
\begin{proof} 
The Lebesgue constant $\| D_{R,d}\|_1$ is well known to be bounded from and below by constants times $R^{\frac{d-1}{2}}$ by work of bound of Ilyin in \cite{Il68}. We refer to Liflyand's survey paper \cite[pp. 17--18]{Li06} which gives a clean and simple proof of this assertion. In fact, the proof in \cite{Li06} yields the upper bound $C_p R^{\frac{d-1}{2}}$ for $\| D_{R,d} \|_{p}$ in the range $1 \leq p<\frac{2d}{d+1}$. This means that we may employ H\"{o}lder's inequality to obtain the desired lower bound for $0<p<1$.
\end{proof}

\begin{proof}[Proof of Theorem~\ref{thm:qto2}]
	We replace \cite[Thm.~2.4]{KQSS22} by Lemma~\ref{lem:smallp} and find, following the proof of \cite[Thm.~2.1]{KQSS22} word for word, that for any fixed $0<p<1$ and $1<q<2$, there exist a positive integer $d$ and a function $h$ is $L^q(\mathbb T^d)$ such that $\| P_+ h \|_p>\| h\|_q$. The conclusion now follows since $p$ can be chosen as small as we please.
\end{proof}

It would be interesting to know if Lemma~\ref{lem:smallp} could be improved. Would it for instance be possible to compute the precise asymptotic behavior of $\| D_{R,d} \|_{0}$? If we were in the favorable situation that this quantity is bounded below by a constant times $R^{\frac{d-o(1)}{2}}$ when $d\to \infty$, then we would be able to prove that $\mathfrak{p}_d(q)=-1$ for all $1<q<2$ when $d$ is sufficiently large. This would have the following consequence: If $q\mapsto \mathfrak{p}_d(q)$ is strictly increasing on $[\mathfrak{q}(d), \infty)$, then the second part of the weak critical  exponent conjecture would hold.

\bibliography{cerp}
\end{document}